\documentclass[11pt]{amsart} \usepackage{amsmath, amssymb}
\usepackage{amsfonts} \usepackage{graphicx} \usepackage{mathrsfs}
\usepackage{enumerate}
\usepackage{color}\usepackage{hyperref}
\usepackage[arrow,matrix,curve,cmtip,ps]{xy}
\usepackage{enumitem}
\usepackage{mathtools}

\allowdisplaybreaks

\newcommand{\KKK}{\mathbb{K}}
\newcommand{\rV}{R}
\newcommand{\hhh}{h}
\newcommand{\bimG}{\mathcal X_E}
\newcommand{\UU}[1][\xx]{\mathsf{U}_{#1}}
\newcommand{\VV}[1][\xx]{\mathsf{V}_{#1}}
\newcommand{\PP}[1][\xx]{\mathsf{P}_{#1}}
\newcommand{\VVt}{\widetilde{\mathsf{V}}_\xx}
\newcommand{\PPt}{\widetilde{\mathsf{P}}_\xx}
\newcommand{\UUt}{\widetilde{\mathsf{U}}_\xx}
\newcommand{\matrM}{\mathsf M}
\newcommand{\AAA}{C^*(E)}
\newcommand{\III}{J}
\newcommand{\ee}{{\mathbf e}}
\newcommand{\xx}{{\mathbf x}}
\newcommand{\yy}{{\mathbf y}}
\newcommand{\zz}{{\mathbf z}}
\newcommand{\Z}{\mathbb{Z}} 
\newcommand{\ZZ}{\mathbb{Z}}
\newcommand{\upindex}[1][\xx]{L^+_{#1}}
\newcommand{\downindex}[1][\xx]{L^-_{#1}}
\newcommand{\wev}{s(e)=w,r(e)=v} 
\newcommand{\lwevprime}{ s(e')=w,r(e')=v'}  
\newcommand{\upm}[1]{\left\lbrack #1\right\rbrack}
\newcommand{\downm}[1]{\left\langle #1\right\rangle}
\newcommand{\EE}[2]{\mathsf{E}_{#1,#2}}

\newenvironment{lbmatrix}{\left[\begin{smallmatrix}}{\end{smallmatrix}\right]}

\DeclareMathOperator{\coker}{coker}
\DeclareMathOperator{\im}{Im}

\theoremstyle{plain}
\newtheorem{theorem}{Theorem}[section]
\newtheorem{lemma}[theorem]{Lemma}
\newtheorem{corollary}[theorem]{Corollary}
\newtheorem{claim}[theorem]{Claim}
\newtheorem{prop}[theorem]{Proposition}
\newtheorem{fact}[theorem]{Fact}
\theoremstyle{remark}
\newtheorem{remark}[theorem]{Remark}

\theoremstyle{definition}
\newtheorem{definition}[theorem]{Definition}
\newtheorem{examp}[theorem]{Example}

\numberwithin{equation}{section}

\begin{document} 

\title{Index maps in the $K$-theory of graph algebras}

\author{Toke Meier Carlsen}

\author{S\o ren Eilers}

\author{Mark Tomforde} 

\address{Department of Mathematical Sciences\\
Norwegian University of Science and Technology\\
NO-7491 Trondheim\\Norway}
\email{tokemeie@math.ntnu.no}

\address{Department for Mathematical Sciences\\University of
Copenhagen\\Universitetsparken 5\\DK-2100 Copenhagen \O \\Denmark}
\email{eilers@math.ku.dk}

\address{Department of Mathematics \\ University of Houston \\
Houston, TX 77204-3008 \\USA} \email{tomforde@math.uh.edu}

\thanks{This research was supported by the NordForsk Research Network
  ``Operator Algebras and Dynamics'' (grant \#11580). The first named author was supported by the Research Council of Norway. The third author was supported by NSA Grant H98230-09-1-0036.}

\date{\today}



\begin{abstract} Let $C^*(E)$ be the graph $C^*$-algebra associated to
  a graph $E$ and let $\III$ be a gauge-invariant ideal in $C^*(E)$.   
  We compute the cyclic six-term exact sequence in $K$-theory
  associated to the extension 
  \begin{equation*}
    0 \longrightarrow \III \longrightarrow C^*(E) \longrightarrow C^*(E)
  / \III \longrightarrow 0
  \end{equation*}
  in terms of the adjacency matrix associated to $E$.  The ordered 
  six-term exact sequence is a complete stable isomorphism invariant
  for se\-ve\-ral classes of graph $C^*$-algebras, for instance those  containing a unique proper nontrivial
  ideal. Further, in many other cases, finite collections of such sequences comprise complete invariants.

Our results allow for explicit computation of the
  invariant, giving an exact sequence in terms of
  kernels and cokernels of matrices determined by the vertex matrix of
  $E$. 
\end{abstract}

\maketitle

\section{Introduction}

The cyclic six-term exact sequence
\begin{equation} \label{eq:1} 
\begin{split}
\xymatrix{
{K_0(\III)}\ar[r]^-{\iota_*}&{K_0\bigl(\AAA\bigr)}\ar[r]^-{\pi_*}&{K_0\bigl(\AAA/\III\bigr)}\ar[d]^-{\partial_0}\\
{K_1\bigl(\AAA/\III\bigr)}\ar[u]^-{\partial_1}&{K_1\bigl(\AAA\bigr)}\ar[l]^-{\pi_*}&{K_1(\III)}\ar[l]^-{\iota_*}}
\end{split}
\end{equation}
is a complete stable isomorphism invariant for a graph $C^*$-algebra
$\AAA$ of real rank zero containing a proper nontrivial ideal $\III$ when any of the following are satisfied
\begin{itemize}
\item $\III$ is the  unique proper nontrivial ideal of $\AAA$ (\cite[Theorem 4.5]{semt:cnga}),
\item $\III$ is a smallest proper nontrivial ideal of $\AAA$, and $\AAA/\III$ is AF  (\cite[Corollary 6.4]{segrer:ccfis}),
\item $\III$ is a largest proper nontrivial ideal of $\AAA$, and $\III$ is AF  (\cite[Theorem 4.7]{semt:cnga}).
\end{itemize}
In other cases (cf. \cite{segrer:ccfis}) a complete invariant may be obtained by combining several six-term exact sequences associated to $\AAA$ and its ideals.

It is therefore important to address how to compute sequences of the form in \eqref{eq:1}. In the
existing literature it is shown that if $E$ is a row-finite graph with
no sinks, then 
\begin{equation*}
K_0\bigl(C^*(E)\bigr) \cong \coker (A^t-I)\text{ and }K_1\bigl(C^*(E)\bigr)
\cong \ker (A^t-I), 
\end{equation*}
where $A^t - I : \Z^{E^0} \to \Z^{E^0}$ is the
linear map given by the transpose of the vertex matrix $A$ of $E$ minus the identity
matrix $I$.  
This description of the $K_0$-group also includes a description of its order, and a similar computation exists when sinks and infinite emitters 
are allowed. 
Since gauge-invariant ideals of graph $C^*$-algebras and the
corresponding quotients are
naturally isomorphic to graph $C^*$-algebras, this allows one to
compute the $K_0$-groups and $K_1$-groups in the above exact sequence.
Moreover, since the $C^*$-algebra of a graph satisfying Condition~(K)
has real rank zero \cite[Theorem~3.5]{jaj:rrcag}, it follows from
\cite{lgbgkp:crrz} that the descending connecting map
$\partial_0 : K_0(\AAA/\III) \to K_1(\III)$ is the zero map.  All that remains is to describe a method for computing the other connecting group homomorphisms.

The purpose of this paper is to provide explicit formulae for computing the six-term exact sequence, the main challenge being to compute the connecting map $\partial_1 : K_1(\AAA/\III) \to
K_0(\III)$.
We shall also show that $\partial_0 : K_0(\AAA/\III) \to K_1(\III)$ is the
zero map regardless of whether the graph $E$ satisfies Condition~(K)
or not. All our calculations hold for an arbitrary graph algebra $\AAA$ and an arbitrary
gauge-invariant ideal $\III$ in $\AAA$, even in the case of so-called \emph{breaking vertices}. 

To compute $\partial_1$, we need to choose generators for the $K$-groups involved.
There is a canonical (and well-known) way to do this in $K_0$; one can
choose an isomorphism of $K_0(C^*(E))$ with $\coker (A^t-I)$ taking
$[p_v]$ to $\ee_v+\im (A^t-I)$, where $\ee_v$ is the vector with a $1$ in
the $v$\textsuperscript{th} position and zeroes elsewhere.  However,
for the $K_1$-group the calculation is substantially harder. 
Descriptions of $K_1$ can be found in \cite{tbdpirws:crg} and
\cite{ddmt:ckegc}, but we need a more explicit description and
therefore choose a different approach, choosing
explicit generators for $K_1$ based on a slightly intricate indexing of the
entries in a matrix over $C^*(E)$.  
Although any quotient of a graph $C^*$-algebra by a gauge-invariant
ideal is isomorphic to a graph $C^*$-algebra, it will be more
convenient for us to use that such a quotient is isomorphic to a
relative graph $C^*$-algebra (cf. \cite{psmmt:atc}), and we will
therefore find generators of $K_0$ and $K_1$, not just
for graph $C^*$-algebras, but for relative graph $C^*$-algebras.

We prove that the generators we choose for $K_1$ are indeed generators
by computing the index map of the canonical Toeplitz extension
of $C^*(E)$, using methods developed by Katsura in that framework. 
Our approach involves computing the index map using the canonical method
(cf.~\cite{mrflnl:ikc}) of lifting the generating unitaries to partial
isometries and computing defects.  This method has similarities with the
approach for Cuntz-Krieger algebras outlined by Cuntz himself
in \cite{jc:hgsec}, and discussed with a few more details in
\cite{mr:ccka}.
After describing how to choose generators for $K_0$ and $K_1$ of
any relative graph $C^*$-algebra, we determine the index map $\partial_1 :K_1(\AAA/\III)\to
K_0(\III)$ by, in a new extension, again lifting our generating
unitaries to partial isometries, and computing defects.

In Section 2 we briefly introduce graph $C^*$-algebras, relative graph
$C^*$-algebras, and gauge-invariant ideals of graph $C^*$-algebras. In
Section 3 we find generators of $K_0$ and $K_1$ of any relative graph
$C^*$-algebra. Section 4 states the main result of the paper, allowing the computation of the index map $\partial_1 :K_1(\AAA/\III)\to
K_0(\III)$ and the other maps in the six-term exact sequence \eqref{eq:1}, and this result is proved in Section 5.

\section{Preliminaries}

A
(directed) graph $E=(E^0, E^1, r, s)$ consists of a countable set
$E^0$ of vertices, a countable set $E^1$ of edges, and maps $r,s: E^1
\rightarrow E^0$ identifying the range and source of each edge.  A
vertex $v \in E^0$ is called a \emph{sink} if $|s^{-1}(v)|=0$, and $v$
is called an \emph{infinite emitter} if $|s^{-1}(v)|=\infty$. A graph
$E$ is said to be \emph{row-finite} if it has no infinite emitters. If
$v$ is either a sink or an infinite emitter, then we call $v$ a
\emph{singular vertex}.  We write $E^0_\textnormal{sing}$ for the set
of singular vertices.  Vertices that are not singular vertices are
called \emph{regular vertices} and we write $E^0_\textnormal{reg}$ for
the set of regular vertices.

If $E$ is a graph, a \emph{Cuntz-Krieger $E$-family} is a set of
mutually orthogonal projections $\{p_v : v \in E^0\}$ and a set of
partial isometries $\{s_e : e \in E^1\}$ with mutually orthogonal ranges which
satisfy the \emph{Cuntz-Krieger relations}:
\begin{enumerate}[leftmargin=*,widest=(CK2)]
\item[(CK1)] $s_e^* s_e = p_{r(e)}$ for every $e \in E^1$;
\item[(CK2)] $p_v = \sum_{s(e)=v} s_e s_e^*$ for every $v \in E^0_\textnormal{reg}$;
\item[(CK3)] $s_e s_e^* \leq p_{s(e)}$ for every $e \in E^1$.
\end{enumerate} The \emph{graph algebra $C^*(E)$} is defined to be the
$C^*$-algebra generated by a universal Cuntz-Krieger $E$-family.

It will in this paper also be relevant to work with \emph{relative
  graph $C^*$-algebras} introduced
in \cite{psmmt:atc}. To define a relative graph $C^*$-algebra we must, in
addition to a graph $E$, specify a subset $\rV$ of
$E^0_\textnormal{reg}$. A \emph{Cuntz-Krieger $(E,\rV)$-family} is then
a set of
mutually orthogonal projections $\{p_v : v \in E^0\}$ and a set of
partial isometries $\{s_e : e \in E^1\}$ with mutually orthogonal ranges which
satisfy the relations (CK1) and (CK3) above together with the
following relative Cuntz-Krieger relation:
\begin{enumerate}[leftmargin=*,widest=(RCK2)]
\item[(RCK2)] $p_v = \sum_{s(e)=v} s_e s_e^*$ for every $v \in \rV$.
\end{enumerate} 
The \emph{relative graph algebra $C^*(E,\rV)$} is defined to be the
$C^*$-algebra generated by a universal Cuntz-Krieger $(E,\rV)$-family.
If $\rV=E^0_\textnormal{reg}$, then a Cuntz-Krieger $(E,\rV)$-family is
the same as a Cuntz-Krieger $E$-family and $C^*(E,\rV)=C^*(E)$. If
$\rV=\emptyset$, then $C^*(E,\rV)$ is the Toeplitz algebra
$\mathcal{T}(E)$ defined in
\cite[Theorem 4.1]{njfir:thb}. We will call a Cuntz-Krieger
$(E,\emptyset)$-family a \emph{Toeplitz-Cuntz-Krieger
  $E$-family}.

A \emph{path} in $E$ is a sequence of edges $\alpha = \alpha_1
\alpha_2 \ldots \alpha_n$ with $r(\alpha_i) = s(\alpha_{i+1})$ for $1
\leq i < n$, and we say that $\alpha$ has length $|\alpha| = n$.  We
let $E^n$ denote the set of all paths of length $n$, and we let $E^*
:= \bigcup_{n=0}^\infty E^n$ denote the set of finite paths in $E$.
Note that vertices are considered paths of length zero.  The maps
$r,s$ extend to $E^*$, and for $v,w \in E^0$ we write $v \geq w$ if
there exists a path $\alpha \in E^*$ with $s(\alpha)=v$ and $r(\alpha)
= w$.  Also for a path $\alpha := \alpha_1 \ldots \alpha_n$ we define
$s_\alpha := s_{\alpha_1} \ldots s_{\alpha_n}$, and for a vertex $v
\in E^0$ we let $s_v := p_v$.  It is a consequence of the
relations (CK1) and (CK3) that $C^*(E,\rV) = \overline{\textrm{span}} \{
s_\alpha s_\beta^* : \alpha, \beta \in E^* \text{ and } r(\alpha) =
r(\beta)\}$.

We say that a path $\alpha := \alpha_1 \ldots \alpha_n$ of length $1$
or greater is a \emph{cycle} if $r(\alpha)=s(\alpha)$, and we call the
vertex $s(\alpha)=r(\alpha)$ the \emph{base point} of the cycle.  A
cycle is said to be \emph{simple} if $s(\alpha_i) \neq s(\alpha_1)$
for all $1 < i \leq n$.  The following is an important condition in
the theory of graph $C^*$-algebras.

$\text{ }$

\noindent \textbf{Condition~(K)}: No vertex in $E$ is the base point
of exactly one simple cycle; that is, every vertex is either the base
point of no cycles or at least two simple cycles.

$\text{ }$

For any graph $E$ a subset $H \subseteq E^0$ is \emph{hereditary} if
whenever $v, w \in E^0$ with $v \in H$ and $v \geq w$, then $w \in H$.
A hereditary subset $H$ is \emph{saturated} if whenever $v \in
E^0_\textnormal{reg}$ with $r(s^{-1}(v)) \subseteq H$, then $v \in H$.
For any saturated hereditary subset $H$, the \emph{breaking vertices}
corresponding to $H$ are the elements of the set 
\begin{equation*}
B_H := \bigl\{ v \in E^0: |s^{-1}(v)| = \infty \text{ and } 0 < |s^{-1}(v) \cap r^{-1}(E^0
\setminus H)| < \infty \bigr\}.
\end{equation*}
An \emph{admissible pair} $(H, S)$
consists of a saturated hereditary subset $H$ and a subset $S
\subseteq B_H$.  For a fixed graph $E$ we order the collection of
admissible pairs for $E$ by defining $(H,S) \leq (H',S')$ if and only
if $H \subseteq H'$ and $S \subseteq H' \cup S'$. For any admissible
pair $(H,S)$ we define
$
\III_{(H,S)}$ to be  the ideal in $C^*(E)$ generated by 
\begin{equation*}
\{p_v : v \in H\} \cup \{p_{v_0}^H : v_0 \in S\},
\end{equation*}
where $p_{v_0}^H$ is the \emph{gap projection} defined by
\begin{equation*}
p_{v_0}^H := p_{v_0} - \sum_{\substack{s(e) = v_0\\r(e) \notin H}} s_e
s_e^*.
\end{equation*}
Note that the definition of $B_H$ ensures that the sum on the right is
finite.

For any graph $E$ there is a canonical gauge action $\gamma :
\mathbb{T} \to \operatorname{Aut} C^*(E)$ with the property that for
any $z \in \mathbb{T}$ we have $\gamma_z (p_v) = p_v$ for all $v \in
E^0$ and $\gamma_z(s_e) = zs_e$ for all $e \in E^1$.  We say that an
ideal $\III \triangleleft C^*(E)$ is \emph{gauge invariant} if
$\gamma_z(\III) \subseteq \III$ for all $z \in \mathbb{T}$.

There is a bijective correspondence between the lattice of admissible
pairs of $E$ and the lattice of gauge-invariant ideals of $C^*(E)$
given by $(H, S) \mapsto \III_{(H,S)}$
\cite[Theorem~3.6]{tbdpirws:crg}. When $E$ satisfies Condition~(K), all
ideals of $C^*(E)$ are gauge invariant \cite[Corollary~3.8]{tbdpirws:crg}
and the map $(H, S) \mapsto \III_{(H,S)}$ is onto the lattice of ideals
of $C^*(E)$.  When $B_H = \emptyset$, we write $\III_H$ in place of
$\III_{(H, \emptyset)}$ and observe that $\III_H$ equals the ideal generated
by $\{ p_v : v \in H \}$.  Note that if $E$ is row-finite, then $B_H$
is empty for every saturated hereditary subset $H$.

\section{$K$-theory for relative graph algebras}
For a graph $E$, the \emph{adjacency matrix} is the $E^0 \times E^0$ matrix $A_E$ with 
\begin{equation*}
A_E(v,w) := \#\bigl\{ e
\in E^1 : s(e)= v \text{ and } r(e)=w \bigr\}.
\end{equation*}
Note that the entries of
$A_E$ are elements of $\{0, 1, 2, \ldots \} \cup \{ \infty \}$.
Writing the adjacency matrix with respect to the decomposition $E^0 =
E^0_\textnormal{reg} \sqcup E^0_\textnormal{sing}$, where the regular
vertices are listed first, we obtain a (possibly infinite) block
matrix 
\begin{equation*}
A_E = \begin{bmatrix}A&\alpha\\ H&\eta
\end{bmatrix}
\end{equation*}
in which all entries of $A$ and $\alpha$ are finite, but the entries
in $H$ and $\eta$ may be infinite. We will often just substitute
``$*$'' for $H$ and $\eta$, as they turn out to be irrelevant for the
$K$-theory. Indeed, by \cite{tbdpirws:crg} and \cite{ddmt:ckegc} we know that the map
\begin{equation*}
\begin{bmatrix}A^t-I\\\alpha^t
\end{bmatrix}:\ZZ^{E^0_\textnormal{reg} }\to \ZZ^{E^0}
\end{equation*}
contains the needed information, as
\begin{equation*}
K_0\bigl(C^*(E)\bigr)\simeq\coker \begin{bmatrix}A^t-I\\\alpha^t
\end{bmatrix}
\qquad
K_1\bigl(C^*(E)\bigr)\simeq\ker \begin{bmatrix}A^t-I\\\alpha^t\end{bmatrix}.
\end{equation*}
This result can be generalized to relative graph $C^*$-algebras.  In fact, we
prove in Proposition~\ref{prop:k-groups} that if $E$ is a graph,
$\rV\subseteq E_\textnormal{reg}$, 
and
$
A_E = \begin{lbmatrix}A&\alpha\\ H&\eta
\end{lbmatrix}
$
is the adjacency matrix of $E$ written with respect to the decomposition
$E^0=\rV\sqcup (E^0\setminus \rV)$, where the vertices belonging to $\rV$ are
listed first, then there exists a group isomorphism 
$\chi_0:\coker \begin{lbmatrix}A^t-I\\\alpha^t\end{lbmatrix}\to
K_0\bigl(C^*(E,\rV)\bigr)$ given for any $v\in \rV$ by 
\begin{equation*}
\chi_0\left(\ee_v+\operatorname{im} \begin{lbmatrix}A^t-I\\\alpha^t\end{lbmatrix}\right)= [p_v]_0,
\end{equation*}
and we construct a similar group
isomorphism $\chi_1$ between
$\ker \begin{lbmatrix}A^t-I\\\alpha^t\end{lbmatrix}$ and $K_1\bigl(C^*(E,\rV)\bigr)$.
For this we first introduce some notation:

Given $\xx\in\ker \begin{lbmatrix}A^t-I\\\alpha^t\end{lbmatrix}$,
first note that by definition $\xx$ has only finitely many nonzero
entries $x_{v_1},\dots x_{v_k}$. We define 
\begin{align*}
\upindex &:= \bigl\{ (e,i) : e \in E^1, 1 \leq i \leq -x_{s(e)} \} 
\cup \{ (v,i) : v \in E^0, 1 \leq i \leq x_v \bigr\}\\
\downindex &:= \bigl\{ (e,i) : e \in E^1, 1 \leq i \leq x_{s(e)} \}
\cup \{ (v,i) : v \in E^0, 1 \leq i \leq -x_v \bigr\}
\end{align*}
and note, using the convention that $r(v)=v$ for any $v\in E^0$, that
\begin{lemma}\label{matchup}
When $\xx\in \ker  \begin{lbmatrix}A^t-I\\\alpha^t\end{lbmatrix}$, then  
for any vertex $v\in E^0$ the sets
\begin{align*}
\upindex[v]&=\bigl\{(x,i)\in \upindex: r(x)=v\bigr\}
\shortintertext{and} 
\downindex[v]&=\bigl\{(x,i)\in \downindex: r(x)=v\bigr\}
\end{align*}
are finite and have the same number of elements. 
\end{lemma}
\begin{proof}
We need to consider three cases separately.

\noindent{\textsc{Case I:} $v\in \rV$ and $x_v\geq 0$.}

\noindent The number of elements in $\upindex[v]$ is
\begin{equation*}
x_v+\sum_{x_w<0}
\# \bigl\{ e \in E^1 : s(e) = w, r(e) = v \bigr\} \cdot (-x_w) \ = \ x_v-\sum_{x_w<0}A^t_{v,w} \ x_w
\end{equation*}
and the number of elements in $\downindex[v]$
is
\begin{equation*}
\sum_{x_w>0}  \# \bigl\{ e \in E^1 : s(e) = w, r(e) = v \bigr\} \cdot x_w \ = \ \sum_{x_w>0}A^t_{v,w} \ x_w
\end{equation*}
so the claim follows by inspecting the $v$ coordinate of the equality $A^t\xx=\xx$.

\noindent{\textsc{Case II:} $v\in \rV \text{ and } x_v<0$.}\\
As above.

\noindent{\textsc{Case III:} $v\in E^0\setminus \rV$.}\\
The number of elements in $\upindex[v]$ is
\begin{equation*}
\sum_{x_w<0}  \# \bigl\{ e \in E^1 : s(e) = w, r(e) = v \bigr\} \cdot (-x_w) \ = \ -\sum_{x_w<0}\alpha^t_{v,w} \ x_w
\end{equation*}
and the number of elements in $\downindex[v]$
is
\begin{equation*}
\sum_{x_w>0} \# \bigl\{ e \in E^1 : s(e) = w, r(e) = v \bigr\} \cdot x_w \ = \ \sum_{x_w>0}\alpha^t_{v,w} \ x_w
\end{equation*}
so the claim follows by inspecting the $v$ coordinate of the equality $\alpha^t\xx=0$.
\end{proof}

\begin{lemma} \label{equal}
$\upindex$ and $\downindex$ are finite sets, and have the same number of elements.
\end{lemma}
\begin{proof}
This follows from Lemma~\ref{matchup}, as indeed $\upindex[v]\not=\emptyset$ only when $v$ lies in the set
\begin{equation*}
\{v : x_v\not=0\}\cup \bigl\{v : x_w \neq 0 \text{ for some } w \in s(r^{-1}(v)) \bigr\}
\end{equation*}
which is finite since no $w$ is an infinite emitter.
\end{proof}

Denote the common number of elements in $\upindex$ and $\downindex$ by $\hhh$. Because of Lemma~\ref{matchup}, we can define bijections
\begin{equation*}
\upm{\cdot}:\upindex\to \{1,\dots,\hhh\}\qquad
\downm{\cdot}:\downindex\to \{1,\dots,\hhh\}
\end{equation*}
with the property that
\begin{equation}\label{goodchoice}
\upm{x,i}=\downm{y,j}\Longrightarrow r(x)=r(y)
\end{equation}
with the convention $r(v)=v$. 

When $\mathfrak{A}$ is a $C^*$-algebra then we let
$\matrM_{\hhh}(\mathfrak{A})$ denote the $C^*$-algebra of
$\hhh\times\hhh$-matrices over $\mathfrak{A}$.
We are ready for our key definitions:

\begin{definition}\label{key}
Suppose that $\mathfrak{A}$ is a $C^*$-algebra which contains  a
Toeplitz-Cuntz-Krieger $E$-family $\{p_v : v \in
E^0\}\cup\{s_e : e \in E^1\}$.
With notation as above, we define the two elements $V,P\in \matrM_{\hhh}(\mathfrak{A})$ by
\begin{eqnarray*}
V&=&\sum_{\substack{ 1\leq i\leq x_w\\  s(e)=w}}
s_e \, \EE{\upm{w,i}}{\downm{e,i}}\ +
\sum_{\substack{ 1\leq i\leq -x_w\\   s(e) = w
  }} s_e^* \, \EE{\upm{e,i}}{\downm{w,i}}
\shortintertext{and}
P&=&\sum_{1\leq i\leq x_w} p_w\EE{\upm{w,i}}{\upm{w,i}}\ +
\sum_{\substack{ 1\leq i\leq -x_w\\ \mathclap{  s(e) = w, r(e) = v}
}}
p_v\EE{\upm{e,i}}{\upm{e,i}}.
\end{eqnarray*}
Here $\EE{\bullet}{\bullet}$ denote the standard matrix units in
$\matrM_\hhh(M(\mathfrak{A}))$ where $M(\mathfrak{A})$ is the
multiplier algebra of $\mathfrak{A}$.
\end{definition}

\begin{lemma}\label{foureqs}
If $\{ s_e, p_v : e \in E^1, v \in E^0\}$ is a Toeplitz-Cuntz-Krieger $E$-family, then
\begin{eqnarray}\label{myp}
P&=&\sum_{1\leq i\leq -x_w} p_w\EE{\downm{w,i}}{\downm{w,i}}\ +
\sum_{\substack{ 1\leq i\leq x_w\\\mathclap{ s(e) = w, r(e) = v}}}
p_v\EE{\downm{e,i}}{\downm{e,i}}, \\\label{myvstar}
V^*&=&\sum_{\substack{ 1\leq i\leq x_w\\  s(e) = w }} s_e^* \EE{\downm{e,i}}{\upm{w,i}}\ +
\sum_{\substack{ 1\leq i\leq -x_w\\   s(e) = w }} s_e \EE{\downm{w,i}}{\upm{e,i}},\\\label{myvvstar}
VV^*&=&\sum_{\substack{ 1\leq i\leq x_w\\   s(e) = w }}
s_es_e^*\EE{\upm{w,i}}{\upm{w,i}}\ +
\sum_{\substack{ 1\leq i\leq -x_w\\\mathclap{s(e) = w, r(e) = v }}}
  p_{v}\EE{\upm{e,i}}{\upm{e,i}},\\\label{myvstarv}
V^*V&=&\sum_{\substack{ 1\leq i\leq -x_w\\   s(e) = w }}
s_es_e^*\EE{\downm{w,i}}{\downm{w,i}}\ +
\sum_{\substack{ 1\leq i\leq x_w\\\mathclap{s(e) = w, r(e) = v }}}
 p_{v}\EE{\downm{e,i}}{\downm{e,i}}.
\end{eqnarray}
\end{lemma}
\begin{proof}
It follows from Lemma \ref{equal} and Equation \eqref{goodchoice} that 
\begin{equation*}
\sum_{(x,i)\in \upindex}p_{r(x)}\EE{\upm{x,i}}{\upm{x,i}}\
=\sum_{(x,i)\in \downindex}p_{r(x)}\EE{\downm{x,i}}{\downm{x,i}},
\end{equation*}
and it is easy to check that 
\begin{equation*}
P\ =\sum_{(x,i)\in\upindex}p_{r(x)}\EE{\upm{x,i}}{\upm{x,i}}
\end{equation*} 
and that 
\begin{equation*}
\sum_{(x,i)\in \downindex}p_{r(x)}\EE{\downm{x,i}}{\downm{x,i}}\ =
\sum_{1\leq i\leq -x_w} p_w\EE{\downm{w,i}}{\downm{w,i}}\ +
\sum_{\substack{ 1\leq i\leq x_w\\\mathclap{s(e) = w, r(e) = v}}}
p_v\EE{\downm{e,i}}{\downm{e,i}} 
\end{equation*}
from which Equation \eqref{myp} then follows.
Equation \eqref{myvstar} is straightforward to check. For Equation \eqref{myvvstar},
using only \eqref{myvstar} and the matrix unit relations we get that 
\begin{equation*}
VV^*\ =\sum_{\substack{ 1\leq i\leq x_w\\ s(e) = w }}
s_es_e^*\EE{\upm{w,i}}{\upm{w,i}}\ +
\sum_{\substack{ 1\leq i\leq -x_w\\\mathclap{\wev} \\\mathclap{\lwevprime}}}
   s^*_es_{e'}\EE{\upm{e,i}}{\upm{e',i}}
\end{equation*}
and \eqref{myvvstar} holds from (CK1) and the fact that the $s_e$'s
have mutually or\-tho\-go\-nal ranges. The computation for $V^*V$ is
similar. 
\end{proof}

\begin{lemma} \label{V-pi-lem}
If $\{ s_e, p_v : e \in E^1, v \in E^0\}$ is a Toeplitz-Cuntz-Krieger $E$-family, then $V$ is a partial isometry with $PV=VP=V$.
\end{lemma}
\begin{proof}
Using Equation \eqref{myvstarv}, the definition of $V$, and the fact
that the $s_e$'s are partial isometries, we see that $VV^*V = V$, so
that $V$ is a partial isometry.  Furthermore, (CK3) implies $PV=V$ and
$VP=V$ by Equation \eqref{myp}.
\end{proof}

We now let $\{ s_e, p_v : e \in E^1, v \in E^0\}$ be the universal
Cuntz-Krieger $(E,\rV)$-family generating $C^*(E,\rV)$ and 
write $\VV$ and $\PP$ for the corresponding elements $V$ and $P$ in
$\matrM_\hhh\bigl(C^*(E,\rV)\bigr)$ defined in Definition \ref{key}, using the
added subscript to emphasize the dependence of each of $V$ and $P$ on
$\xx\in\ker \begin{lbmatrix}A^t-1\\\alpha^t\end{lbmatrix}$.  In
addition, we define $\UU :=\VV+(1-\PP)$. 

\begin{fact}\label{ispi}
We have that $\VV\VV^*=\VV^*\VV=\PP$,
and hence that $\UU$ is a unitary.
\end{fact}
\begin{proof}
It follows from Equation \eqref{myvvstar} and (RCK2) that
\begin{eqnarray*}
\VV\VV^*&=&
\sum_{1\leq i\leq x_w}\left( \sum_{s(e) = w}
  s_es_e^*\right)\EE{\upm{w,i}}{\upm{w,i}}\ +
\sum_{\substack{ 1\leq i\leq -x_w\\\mathclap{\wev}}}
  p_{v}\EE{\upm{e,i}}{\upm{e,i}}\\
&=&\sum_{1\leq i\leq x_w}p_w\EE{\upm{w,i}}{\upm{w,i}}\ +
\sum_{\substack{ 1\leq i\leq -x_w\\\mathclap{\wev}}}
  p_{v}\EE{\upm{e,i}}{\upm{e,i}} \\
&=& \PP
\end{eqnarray*}
showing the first claim. Likewise, Equation \eqref{myvstarv} and (RCK2) show that $\VV^*\VV=\PP$.  The fact that $\UU$ is a unitary follows.
\end{proof}

\begin{remark}\label{indep}
Notice that although $\UU$ does depend on the choice of bijections
\begin{equation*}
\upm{\cdot}:\upindex\to \{1,\dots,\hhh\}\qquad
\downm{\cdot}:\downindex\to \{1,\dots,\hhh\},
\end{equation*}
the element $[\UU]_1$ of $K_1\bigl(C^*(E,\rV)\bigr)$ does not.
\end{remark}

\begin{prop} \label{prop:k-groups}
Let $E$ be a graph, let $V$ be a subset of $E_\textnormal{reg}$
and let
\begin{equation*}
A_E = \begin{bmatrix}A&\alpha\\ H&\eta
\end{bmatrix}
\end{equation*}
be the adjacency matrix of $E$ written with respect to the decomposition
$E^0=V\sqcup (E^0\setminus V)$ where the vertices belonging to $V$ are
listed first.
\begin{enumerate}
\item There exists a group isomorphism
$\chi_0:\coker \begin{lbmatrix}A^t-I\\\alpha^t\end{lbmatrix}\to
K_0\bigl(C^*(E,\rV)\bigr)$ given for any $v\in E^0$ by 
\begin{equation} \label{eq:2}
\chi_0\left(\ee_v+\operatorname{im} \begin{bmatrix}A^t-I\\\alpha^t\end{bmatrix}\right)= [p_v]_0.
\end{equation}
The preimage of the positive cone of $K_0\bigl(C^*(E,\rV)\bigr)$ is generated by
\begin{equation*}
\Bigl\{\ee_v: v\in E^0\Bigr\}\cup 
\Bigl\{\ee_v-\sum _{e\in F}\ee_{r(e)}: v\in E^0_\textnormal{sing},\ 
F\subseteq s^{-1}(v),\ F\text{ finite}\Bigr\}. 
\end{equation*}
\item The map
  $\chi_1:\ker \begin{lbmatrix}A^t-I\\\alpha^t\end{lbmatrix}\to
  K_1\bigl(C^*(E,\rV)\bigr)$ given by 
\begin{equation*}
\chi_1(\xx)=[\UU]_1
\end{equation*}
is group isomorphism.
\end{enumerate}
\end{prop}
\begin{proof}
  As noted in \cite{psmmt:atc}, we can realize $C^*(E,\rV)$ as a relative
  Cuntz-Pimsner algebra over a Hilbert bimodule $\bimG$. 
  It is not difficult to check that the corresponding Toeplitz algebra
  $\mathcal{T}_{\bimG}$ is isomorphic to the Toeplitz algebra
  $\mathcal{T}(E)$.
  We let $\pi : \mathcal{T} (E) \to C^*(E,\rV)$ denote the canonical map, so that
\begin{equation} \label{eq:ses}
\xymatrix{0\ar[r] & \ker \pi \ar[r]^\iota 
  & \mathcal{T}( E) \ar[r]^{\pi} 
  & C^*(E,\rV) \ar[r] &0}
\end{equation}
is exact.  The associated $K$-theory is then
\begin{equation*}
  \xymatrix{
    K_0(\ker \pi)
    \ar[r]^{\iota_*} &
    K_0\bigl(\mathcal{T}(E)\bigr)
    \ar[r]^-{\pi_*} & 
    K_0 \bigl(C^*(E,\rV)\bigr)
    \ar[d]\\ 
    K_1 \bigl(C^*(E,\rV)\bigr) 
    \ar[u]^{\partial_1} & 
    K_1 \bigl(\mathcal{T}(E)\bigr)
    \ar[l]^-{\pi_*}&
    K_1(\ker \pi). 
    \ar[l]^-{\iota_*}}
\end{equation*}

Now we appeal to Katsura's work. It follows from the results of
\cite[\S8]{tk:cac}, that $\ker\pi$ and $\mathcal{T}(E)$ are
$KK$-equivalent to the commutative $AF$-algebras $c_0(\rV)$ and
$c_0(E^0)$, respectively, and that there are group isomorphisms $\kappa : K_0(\ker \pi) \to \ZZ^{\rV}$ and $\lambda : K_0\bigl(\mathcal{T} (E)\bigr) \to \ZZ^{E^0}$ such that the diagram
\begin{equation*}
\xymatrix@C-3mm{
0\ar[r]&{K_1\bigl(C^*(E,\rV)\bigr)}\ar[r]^-{\partial_1 }&{K_0(\ker \pi)}\ar[r]^-{\iota_*}\ar[d]_-{\kappa} & {K_0\bigl(\mathcal T(E)\bigr)}\ar[d]_-{\lambda}\ar[r]^-{\pi_*}&{K_0\bigl(C^*(E,\rV)\bigr)}\ar[r]&0 \\ 
&&{\ZZ^{\rV}}\ar[r]_{\begin{lbmatrix}A^t-I\\\alpha^t  \end{lbmatrix}}&{\ZZ^{E^0}}}
\end{equation*}
commutes with the top row exact. In \cite{tk:cac} concrete
$*$-homomorphisms are given inducing $\kappa$ and $\lambda$, but we do
not need them here. All we need is the fact that $\lambda(p_v)=\ee_v$ and
\begin{equation} \label{eq:3}
  \kappa\left(\left[p_w-\sum_{s(e)=w}s_es_e^*\right]_0\right)=\ee_w
\end{equation}
for $v\in E^0$ and $w\in \rV$.
It follows that $\pi_*\circ\lambda^{-1}$ is a surjective group
homomorphism from $\ZZ^{E^0}$ to $K_0\bigl(C^*(E,\rV)\bigr)$ which for any $v\in
E^0$ maps $\ee_v$ to $[p_v]_0$ and whose kernel is
$\operatorname{im} \begin{lbmatrix}A^t-I\\\alpha^t
\end{lbmatrix}$. The existence of a group isomorphism
$\chi_0:\coker \begin{lbmatrix}A^t-I\\\alpha^t\end{lbmatrix}\to
K_0\bigl(C^*(E,\rV)\bigr)$ which for any $v\in E^0$ satisfies Equation
\eqref{eq:2} follows from this. The description of the positive cone
in the row-finite case was given in \cite[Theorem~7.1]{pamamep:nkga}.
For the general situation, it is shown in \cite[Theorem~2.2]{mt:okgc}
that the process of desingularization can be used to extend the result
from the row-finite case to the general case. 

To see that 
$\chi_1:\ker \begin{lbmatrix}A^t-I\\\alpha^t\end{lbmatrix}\to
K_1\bigl(C^*(E,\rV)\bigr)$ is a group isomorphism, fix 
$\xx\in\ker\begin{lbmatrix}A^t-I\\\alpha^t\end{lbmatrix}$
and lift $\UU=\VV+(1-\PP)\in \matrM_\hhh\bigl(C^*(E,\rV)\bigr)$ to
$\UUt=\VVt+(1-\PPt)\in\matrM_\hhh\bigl(\mathcal T(E)\bigr)$
where $\VVt$ and
$\PPt$ are the elements $V$ and $P$ in $\matrM_\hhh\bigl(\mathcal T(E)\bigr)$ we
get by using the universal Toeplitz-Cuntz-Krieger $E$-family which
generates $\mathcal T(E)$ in Definition \ref{key}. 
By Lemma~\ref{V-pi-lem}, $\VVt$ is a partial isometry with $\PPt\VVt=\VVt\PPt=\VVt$.  It follows that $\UUt$ is also a partial isometry.
We need to compute the defect of $\UUt$ as an element of
$K_0(\ker\pi)$. We have by Lemma~\ref{foureqs} that
\begin{equation*}
1-\UUt\UUt^*=\PPt-\VVt\VVt^*\
=\sum_{1\leq i\leq x_w}\left(p_w-\sum_{s(e) = w}s_es_e^*\right)\EE{\upm{w,i}}{\upm{w,i}}
\end{equation*}
and a similar equation for $1-\UUt^*\UUt$. Hence, in $K_0(\ker \pi)$ we have that
\begin{equation}\label{fortakeshi}
\bigl[1-\UUt\UUt^*\bigr]_0-\bigl[1-\UUt^*\UUt\bigr]_0=\sum_{x_w\not=0}x_w\left[p_w-\sum_{s(e)=w}s_es_e^*\right]_0
\end{equation}
which  together with Equation \eqref{eq:3} and Equation
\eqref{fortakeshi} implies that
\begin{equation}\label{leftinv}
\kappa\circ\partial_1\circ \chi_1(\xx)=\xx
\end{equation}
for any $\xx\in
\begin{lbmatrix}A^t-I\\\alpha^t  \end{lbmatrix}$. This shows that
$\chi_1$ is injective.  Let us also prove that $\chi_1$ is a group
isomorphism. Fix $\yy\in K_1(C^*(E,\rV))$ and note that 
\begin{equation*}
\begin{bmatrix}A^t-I\\\alpha^t
 \end{bmatrix}\circ\kappa\circ\partial_1(\yy)=
\lambda\circ \iota_*\circ\partial_1(\yy)=0
\end{equation*}
so that $\zz :=\kappa\circ\partial_1(\yy)$ lies in $\ker
\begin{lbmatrix}A^t-I\\\alpha^t
 \end{lbmatrix}$. Since $\kappa\circ\partial_1$ is injective, it
 follows from Equation \eqref{leftinv} that $\chi_1(\zz)=\yy$. We
 conclude that $\kappa\circ\partial_1$ is actually an inverse to
 $\chi_1$, and hence $\chi_1$ is a group isomorphism. 
\end{proof}

\section{The index map}

\renewcommand{\VVt}{\widehat{\mathsf{V}}_\xx}
\renewcommand{\PPt}{\widehat{\mathsf{P}}_\xx}
\renewcommand{\UUt}{\widehat{\mathsf{U}}_\xx}

Let $E$ be a graph and let $\III$ be a gauge-invariant ideal in
$C^*(E)$. It follows from \cite{tbdpirws:crg} that $\III$ is of the form
$\III_{(H,S)}$ for an admissible pair $(H,S)$.
Writing the adjacency matrix of $E$ with respect to the decomposition 
\begin{equation*}
E^0_\textnormal{reg} \cap H, \quad E^0_\textnormal{sing} \cap H, \quad
E^0_\textnormal{reg} \setminus H, \quad E^0_\textnormal{sing}
\backslash (H\cup S),\quad S
\end{equation*}
we arrive at the matrix
\begin{equation*}
\begin{bmatrix} A&\alpha&0&0&0\\ *&*&0&0&0\\
X&\xi&B&\beta&\eta\\
*&*&*&*&*\\
*&*&\Gamma&\gamma&Z
\end{bmatrix}.
\end{equation*}

We are now ready to state our main result. Here and below, 
whenever $T:G_1\to G_2$ is a group homomorphism between abelian groups
and $H_1$ and $H_2$ are subgroups of $G_1$ and $G_2$, respectively,
such that $T(H_1)\subseteq H_2$, then we also use $T$ to denote the group
homomorphism from $G_1/H_1$ to $G_2/H_2$ induced by $T$, and we denote by
$I_{a_1\cdots a_k}$ the canonical inclusion of the indicated
components of a direct sum into a larger direct sum, and by   
$P_{a_1\cdots a_k}$ the corresponding projection.  

\begin{theorem}\label{main}
Let $E$ be a graph and let $(H,S)$ be an admissible pair.  
The six term exact sequence in $K$-theory 
\begin{equation*}
\xymatrix{
{K_0\bigl(\III_{(H,S)}\bigr)}\ar[r]^-{\iota_*}&{K_0\bigl(C^*(E)\bigr)}\ar[r]^-{\pi_*}&{K_0\bigl(C^*(E) / \III_{(H,S)}\bigr)}\ar[d]^-{\partial_0}\\
{K_1\bigl(C^*(E) /\III_{(H,S)}\bigr)}\ar[u]^-{\partial_1}&{K_1\bigl(C^*(E)\bigr)}\ar[l]^-{\pi_*}&{K_1\bigl(\III_{(H,S)}\bigr)}\ar[l]^-{\iota_*}
}
\end{equation*}
is isomorphic to
\begin{equation*}
\xymatrix{
*[l]{\coker\begin{lbmatrix} A^t-I\\\alpha^t\\0\end{lbmatrix}}
\ar[r]^-{\widetilde{I}}&
{\coker\begin{lbmatrix} A^t-I&X^t\\\alpha^t&\xi^t\\0&B^t-I\\0&\beta^t\\0&\eta^t\end{lbmatrix}}
\ar[r]^-{P_{345}}&
*[r]{\coker\begin{lbmatrix} B^t-I&\Gamma^t\\\beta^t&\gamma^t\\\eta^t&Z^t-I\end{lbmatrix}}
  \ar[d]^-{0}\\
{\ker\begin{lbmatrix} B^t-I&\Gamma^t\\\beta^t&\gamma^t\\\eta^t&Z^t-I\end{lbmatrix}}
\ar[u]^-{\begin{lbmatrix}X^t&0\\\xi^t&0\\0&I\end{lbmatrix}}&
{\ker\begin{lbmatrix} A^t-I&X^t\\\alpha^t&\xi^t\\0&B^t-I\\0&\beta^t\\0&\eta^t\end{lbmatrix}}
\ar[l]^-{I_1\circ P_2}&
*[r]{\ker\begin{lbmatrix} A^t-I\\\alpha^t\\0\end{lbmatrix}}\ar[l]^-{I_1}
}
\end{equation*}
where $\widetilde{I}$ is given by the block matrix
\begin{equation*}
\begin{bmatrix}
I&0&0\\0&I&0\\0&0&-\Gamma^t\\0&0&-\gamma^t\\0&0&I-Z^t
\end{bmatrix}
=I_{125}-
\begin{bmatrix}
0&0&0\\0&0&0\\0&0&\Gamma^t\\0&0&\gamma^t\\0&0&Z^t
\end{bmatrix}.
\end{equation*}
\end{theorem}

Each cokernel is ordered as described in Theorem
\ref{prop:k-groups}. We postpone the proof of the theorem to the
ensuing section, but remark here that the isomorphism between the two
six term exact sequences is given by explicit defined maps which are described
in the proof. 

For now, let us record a number of examples and specializations:

\begin{remark} \label{remark:reduced}
  If the saturated hereditary subset $H$ has no breaking vertices (this is always the case if
  $E$ is row-finite), or if 
  $S=\emptyset$, then the six term exact sequence of Theorem \ref{main}
 reduces to 
  \begin{equation} \label{six-term-reduced}
    \begin{split}
    \xymatrix{
      {\coker\begin{lbmatrix} A^t-I\\\alpha^t\end{lbmatrix}}
      \ar[r]^-{I_{12}}&
      {\coker\begin{lbmatrix} A^t-I&X^t\\\alpha^t&\xi^t\\0&B^t-I\\0&\beta^t\end{lbmatrix}}
      \ar[r]^-{P_{34}}&
      *[r]{\coker\begin{lbmatrix} B^t-I\\\beta^t\end{lbmatrix}}
      \ar[d]^-{0}\\
      {\ker\begin{lbmatrix} B^t-I\\\beta^t\end{lbmatrix}}
      \ar[u]^-{\begin{bmatrix}X^t\\\xi^t\end{bmatrix}}&
      {\ker\begin{lbmatrix} A^t-I&X^t\\\alpha^t&\xi^t\\0&B^t-I\\0&\beta^t\end{lbmatrix}}
      \ar[l]^-{P_2}&
      {\ker\begin{lbmatrix} A^t-I\\\alpha^t\end{lbmatrix}.}\ar[l]^-{I_1}
    }
  \end{split}
  \end{equation}
\end{remark}

\begin{remark}
Let $E$ be a row-finite graph with no sinks. Then any gauge-invariant
ideal in $C^*(E)$ has the form $\III_H$ for some saturated hereditary
subset $H$ and the six term exact sequence
of Theorem \ref{main} reduces in this case to 
\begin{equation*}
\xymatrix{
{\coker\begin{bmatrix} A^t-I\end{bmatrix}}
\ar[r]^-{I_1}&
{\coker\begin{bmatrix} A^t-I&X^t\\ 0&B^t-I\end{bmatrix}}
\ar[r]^-{P_2}&
{\coker\begin{bmatrix} B^t-I\\ \end{bmatrix}}
  \ar[d]^-{0}\\
{\ker\begin{bmatrix} B^t-I\\ \end{bmatrix}}
\ar[u]^-{\begin{bmatrix}X^t \end{bmatrix}}&
{\ker\begin{bmatrix} A^t-I&X^t\\ 0&B^t-I\\ \end{bmatrix}}
\ar[l]^-{P_2}&
{\ker\begin{bmatrix} A^t-I\end{bmatrix}.}\ar[l]^-{I_1}
}
\end{equation*}
\end{remark}

\begin{corollary} \label{cor}
Let $E$ be a graph such that the associated graph $C^*$-algebra $C^*(E)$ contains a unique proper nontrivial ideal.  Then this ideal has the form $\III_H$ for some saturated hereditary subset $H$ with no breaking vertices.  Consequently, the cyclic six term exact sequence determined by the short exact sequence $0 \to \III_H \to C^*(E) \to C^*(E) / \III_H \to 0$ is isomorphic to the cyclic exact sequence described in \eqref{six-term-reduced}.
\end{corollary}

\begin{proof}
If $E$ has a unique proper nontrivial ideal, then it follows from \cite[Lemma~3.1]{semt:cnga} that the ideal has the form $\III_H$ for a saturated hereditary subset $H$ with no breaking vertices.
\end{proof}
\newcommand{\blank}{\makebox[1ex]{\rule{0ex}{1ex}{}}}
\begin{examp} Consider the class of graphs $E_{x,y,z}$ given by the adjacency matrix
\begin{equation*}
\begin{bmatrix}
0&0&0&0\\
x&1&1&0\\
y&1&1&1\\
z&0&1&1
\end{bmatrix}
\end{equation*}
where $x, y, z \in \mathbb{N}$.  These graphs all satisfy
Condition~(K) and have one nontrivial saturated hereditary subset (the
subset consisting of the first vertex).
Thus we are in the situation
of Corollary \ref{cor}, with $E^0_\textnormal{reg}=\{v_2,v_3,v_4\}$
and $E^0_\textnormal{reg}=H=\{v_1\}$. Hence the adjacency matrix has the block form
\begin{equation*}
\left[\begin{array}{c|ccc}
\alpha&&0&\\\hline
\blank&\blank&\blank&\blank\\
\xi&\blank&B&\blank\\
\blank&\blank&\blank&\blank\\
\end{array}
\right]
\end{equation*}
and the six-term exact sequence is
\begin{equation*}
\xymatrix{
{\coker 0_{1\times 0}}
\ar[r]^-{I_1}&
{\coker\begin{lbmatrix} x&y&z\\0&1&0\\1&0&1\\0&1&0\end{lbmatrix}}
\ar[r]^-{P_{234}}&
*[r]{\coker\begin{lbmatrix}0&1&0\\1&0&1\\0&1&0\end{lbmatrix}}
  \ar[d]^-{0}\\
{\ker\begin{lbmatrix}0&1&0\\1&0&1\\0&1&0\end{lbmatrix}}
\ar[u]^-{\begin{lbmatrix}x&y&z\end{lbmatrix}}&
{\ker\begin{lbmatrix} x&y&z\\0&1&0\\1&0&1\\0&1&0\end{lbmatrix}}
\ar[l]^-{P_{123}}&
{\ker 0_{1\times 0}}\ar[l]^-{0}
}
\end{equation*}
which simplifies to
\begin{equation*}
\xymatrix{
0\ar[r]&
{\ZZ}
\ar[r]^-{x-z}&
{\ZZ}
\ar[r]&
{ \ZZ/(x-z)\oplus \ZZ}
\ar[r]&
{\ZZ}
  \ar[r]&0}
\end{equation*}
when $x\not=z$ and to
\begin{equation*}
\xymatrix{
0\ar[r]&
{\ZZ}\ar@{=}[r]&
{\ZZ}
\ar[r]^-{0}&
{\ZZ}\ar[r]&
{\ZZ\oplus \ZZ}
\ar[r]&
{\ZZ}
  \ar[r]&0
}
\end{equation*}
when $z=x$.

The $K_0$-group of the ideal is canonically ordered, and the order of the $K_0$-group of the quotient is trivial, irrespective of $x,y,z$. We may hence apply \cite{semt:cnga} to prove that
$C^*(E_{x,y,z})\otimes \KKK\simeq C^*(E_{x',y',z'})\otimes \KKK$
precisely when $x-z=\pm(x'-z')$.
\end{examp}

\begin{examp}
Consider the class of graphs $F_{y,z}$ given by the adjacency matrix
\begin{equation*}
\begin{bmatrix}
0&0&0\\
y&3&1\\
\infty&z&3
\end{bmatrix}
\end{equation*}
where $ y, z \in \mathbb{N}$.  These graphs all satisfy
Condition~(K) and have one nontrivial saturated hereditary subset
$\{v_1\}$ for which $\{v_3\}$ is breaking. We furthermore have that
$E^0_\textnormal{reg}=\{v_2\}$ and
$E^0_\textnormal{sing}=\{v_1,v_3\}$.
If we consider the ideal
$\III_{(\{v_1\},\{v_3\})}$, then the adjacency matrix has the block form
\begin{equation*}
\left[
\begin{array}{c|c|c}
*&0&0\\\hline
\xi&B&\eta\\\hline
*&\Gamma&Z
\end{array}
\right]
\end{equation*}
which gives
\begin{equation*}
\xymatrix{
{\coker 0_{2\times 0}}\ar[r]^-{\begin{lbmatrix}1&0\\0&-z\\0&-2
  \end{lbmatrix}}
  &
{\coker\begin{lbmatrix}y\\2\\1\end{lbmatrix}}\ar[r]&
{\coker\begin{lbmatrix}2&z\\1&2\end{lbmatrix}}\ar[d]
\\
{\ker\begin{lbmatrix}2&z\\1&2\end{lbmatrix}}\ar[u]^-{\begin{lbmatrix}y&0\\0&1\end{lbmatrix}}&
{\ker\begin{lbmatrix}y\\2\\1\end{lbmatrix}}\ar[l]&
{\ker 0_{2\times 0}}\ar[l]
}
\end{equation*}
simplifying to
\begin{equation*}
\xymatrix@C+3mm{
0\ar[r]&
{\ZZ^2}\ar[r]^{\begin{lbmatrix}-1&-2y\\0&4-z
  \end{lbmatrix}}
  &
{\ZZ^2}\ar[r]&
{\ZZ_{z-4}}\ar[r]&0}
\end{equation*}
when $z\not=4$ and to 
\begin{equation*}
\xymatrix@C+3mm{
0\ar[r]&{\ZZ}\ar[r]^{\begin{lbmatrix}-2y\\1
  \end{lbmatrix}}
&
{\ZZ^2}\ar[r]^{\begin{lbmatrix}-1&-2y\\0&0
  \end{lbmatrix}}
  &
{\ZZ^2}\ar[r]&{\ZZ}\ar[r]&
{0}}
\end{equation*}
when $z=4$.

In both cases, the $K_0$-group of the ideal is ordered by 
\begin{equation*}
\bigl\{(x_1,x_3): x_3>1\text{ or } [x_3=0, x_1\geq 0]\bigr\},
\end{equation*}
having only the trivial automorphism,
so the computations combine with \cite[Theorem 4.7]{semt:cnga} to show that 
$C^*(F_{y,z})\otimes \KKK\simeq C^*(F_{y',z'})\otimes \KKK$
precisely when $4-z=\pm(4-z')$ and $y-y'\in (4-z)\ZZ$.
\end{examp}

\section{Proof of main result}

The isomorphism of the two six-term exact sequences in Theorem
\ref{main} is given by the six group isomorphisms
$\chi_0',\chi_0,\chi_0'',\chi_1',\chi_1,\chi_1''$ defined as follows.   If we let $E_{(H,S)}$ be the subgraph of $E$ with vertices $H\cup S$
and edges $s^{-1}(H)\cup \bigl(s^{-1}(S)\cap r^{-1}(H)\bigr)$, then the graph
$C^*$-algebra $C^*\bigl(E_{(H,S)}\bigr)$ is isomorphic to a full corner of
$\III_{(H,S)}$  via an embedding $\phi : C^*\bigl(E_{(H,S)}\bigr) \to \III_{(H,S)}$
with $\phi(p_v) = p_v$ for $v\in H$, $\phi(p_{v_0})=p^H_{v_0}$ for
$v_0\in S$ and $\phi(s_e ) = s_e$ for $e \in E^1_{(H,S)}$
(cf. \cite{tbdpirws:crg}).
Notice that
$(E_{(H,S)})^0_{\textnormal{reg}}=E^0_{\textnormal{reg}}\cap H$ and
that $(E_{(H,S)})^0_{\textnormal{sing}}=(E^0_{\textnormal{sing}}\cap
H)\cup S$.
It follows (for example by \cite[Proposition~1.2]{wlp:kacgc}) that $\phi$
induces an isomorphism $\phi_* : K_*\bigl(C^*(E_{(H,S)})\bigr) \to
K_*\bigl(\III_{(H,S)}\bigr)$.  
Thus if we let $\chi_*^{E_{(H,S)}}$denote the group isomorphisms given by Proposition
\ref{prop:k-groups} for $C^*(E_{(H,S)})$, then 
\begin{align*}
  \chi_0'&:=\phi_*\circ \chi_0^{E_{(H,S)}}: \coker\begin{lbmatrix} A^t-I\\\alpha^t\\0\end{lbmatrix}\to
  K_0\bigl(\III_{(H,S)}\bigr)
  \shortintertext{and} 
  \chi_1'&:=\phi_*\circ \chi_1^{E_{(H,S)}}: \ker\begin{lbmatrix} A^t-I\\\alpha^t\\0\end{lbmatrix}\to
  K_1\bigl(\III_{(H,S)}\bigr) 
\end{align*}
are group isomorphisms. 
Similarly, if we let $E \setminus H$ be the subgraph of
$E$ with vertices $E^0 \setminus H$ and edges $r^{-1}(E^0\setminus
H)$, then there is an isomorphism $\psi:C^*(E\setminus H,S)\to C^*(E)/\III_{(H,S)}$
which for any $v\in E^0\setminus H$ maps $p_v$ to
$p_v+\III_{(H,S)}$ and for any $e\in r^{-1}(E^0\setminus H)$
maps $s_e$ to $s_e+\III_{(H,S)}$ (cf. \cite[Example 3.10]{psmmt:atc}).
Notice that
$(E\setminus H)^0_{\textnormal{reg}}=E^0_{\textnormal{reg}}\setminus H$ and
that $(E\setminus
H)^0_{\textnormal{sing}}=E^0_{\textnormal{sing}}\setminus H$.
Thus if we let 
$\chi_*^{(E\setminus H,S)}$
denote the group isomorphisms given by Proposition
\ref{prop:k-groups} for $C^*(E\setminus H,S)$, then 
\begin{align*}
  \chi_0''&:=\psi_*\circ\chi_0^{(E\setminus H,S)}: \coker\begin{lbmatrix}
    B^t-I&\Gamma^t\\\beta^t&\gamma^t\\\eta^t&Z^t-I\end{lbmatrix} \to
  K_0\bigl(C^*(E)/\III_{(H,S)}\bigr)
  \shortintertext{and} 
  \chi_1''&:=\psi_*\circ\chi_1^{(E\setminus H,S)}: \ker\begin{lbmatrix}
    B^t-I&\Gamma^t\\\beta^t&\gamma^t\\\eta^t&Z^t-I\end{lbmatrix} \to
  K_1\bigl(C^*(E)/\III_{(H,S)}\bigr) 
\end{align*}
are group isomorphisms. 
Finally we let $\chi_*$
denote the group isomorphisms given directly by Proposition
\ref{prop:k-groups} for $C^*(E)$.

The theorem then follows from the ensuing six claims.

\begin{claim}
  $\iota_*\circ\chi_0'=\chi_0\circ \widetilde{I}$.
\end{claim}

\begin{proof}
  If $v\in H$, then we have that
  \begin{eqnarray*}
    \chi_0\circ \begin{lbmatrix}I&0&0\\0&I&0\\0&0&-\Gamma^t\\0&0&-\gamma^t\\0&0&I-Z^t\end{lbmatrix}
    \left(\ee_v+\operatorname{im}\begin{lbmatrix} A^t-I\\\alpha^t\\0\end{lbmatrix}\right)
    &=&\chi_0\left(\ee_v+\operatorname{im} 
      \begin{lbmatrix}
        A^t-I&X^t\\\alpha^t&\xi^t\\0&B^t-I\\0&\beta^t\\0&\eta^t\end{lbmatrix}\right)\\
    &= &[p_v]_0=\bigl[\iota(\phi(p_v))\bigr]_0\\
    &=&\iota_*\circ\chi_0' 
    \left(\ee_v+\operatorname{im}\begin{lbmatrix} A^t-I\\\alpha^t\\0\end{lbmatrix}\right),
  \end{eqnarray*}
  and if $v_0\in S$, the left hand side equals
  \begin{eqnarray*}
   &&\chi_0\left(\ee_{v_0}-\sum_{\substack{s(e) = v_0\\ r(e) \notin H}} \ee_{r(e)}+\operatorname{im} 
      \begin{lbmatrix}
        A^t-I&X^t\\\alpha^t&\xi^t\\0&B^t-I\\0&\beta^t\\0&\eta^t\end{lbmatrix}\right)\\
    &= &[p_{v_0}]_0-\sum_{\substack{s(e) = v_0\\r(e) \notin
        H}}[s_es_e^*]_0=\bigl[\iota(p_{v_0}^H)\bigr]_0=\bigl[\iota(\phi(p_v))\bigr]_0\\
    &=&\iota_*\circ\chi_0' 
    \left(\ee_v+\operatorname{im}\begin{lbmatrix} A^t-I\\\alpha^t\\0\end{lbmatrix}\right).
  \end{eqnarray*}
\end{proof}

\begin{claim}
  $\pi_*\circ\chi_0=\chi_0''\circ P_{345}$.
\end{claim}

\begin{proof}
As above, we check the claim of each class given by $\ee_v$. If $v\in H$, then both sides vanish. If $v\notin H$, both sides equal $[p_v]_0$.
\end{proof}

\begin{claim} \label{ciii}
  $\pi_*\circ\chi_1=\chi_1''\circ I_1\circ P_2$.
\end{claim}

\begin{proof} Fix 
  \begin{equation*}
    \xx=\begin{lbmatrix}\yy\\\zz \end{lbmatrix}\in {\ker\begin{lbmatrix}
        A^t-I&X^t\\\alpha^t&\xi^t\\0&B^t-I\\0&\beta^t\\0&\eta^t\end{lbmatrix}}.
  \end{equation*} 
  Then $\begin{lbmatrix}\zz\\ 0 \end{lbmatrix}\in 
  \ker\begin{lbmatrix}
    B^t-I&\Gamma^t\\\beta^t&\gamma^t\\\eta^t&Z^t-I\end{lbmatrix}$ and we
  furthermore have that
  $\upindex[\begin{lbmatrix}\zz\\
    0 \end{lbmatrix}]\subseteq\upindex[\xx]$ and $\downindex[\begin{lbmatrix}\zz\\
    0 \end{lbmatrix}]\subseteq\downindex[\xx]$. Thus if we let $\hhh$ be
  the number of elements in $\upindex[\xx]$ (and in
  $\downindex[\xx]$), and we let $\hhh'$ denote the number of elements in $\upindex[\begin{lbmatrix}\zz\\
    0 \end{lbmatrix}]$ (and in $\downindex[\begin{lbmatrix}\zz\\
    0 \end{lbmatrix}]$), then we can choose the bijections 
  \begin{align*}
  &\upm{\cdot}:\upindex\to \{1,\dots,\hhh\}&
  &\downm{\cdot}:\downindex\to \{1,\dots,\hhh\}
  \shortintertext{and}
  &\upm{\cdot}':\upindex[\begin{lbmatrix}\zz\\
    0 \end{lbmatrix}]\to \{1,\dots,\hhh'\}
  &&\downm{\cdot}':\downindex[\begin{lbmatrix}\zz\\
    0 \end{lbmatrix}]\to \{1,\dots,\hhh'\}
  \end{align*}
  such that $\upm{\cdot}$ is an extension of $\upm{\cdot}'$, and
  $\downm{\cdot}$ is an extension of $\downm{\cdot}'$. 
  We then have that
  \begin{eqnarray*}
    \pi(\VV[\xx]) &=&\pi\left(\sum_{\substack{ 1\leq i\leq x_w\\  s(e)=w
        }} s_e \EE{\upm{w,i}}{\downm{e,i}}+
      \sum_{\substack{ 1\leq i\leq -x_w\\  s(e)=w
        }} s_e^* \EE{\upm{e,i}}{\downm{w,i}}\right)\\
    &=&\sum_{\substack{ 1\leq i\leq z_w\\  s(e)=w
      }} \pi(s_e) \EE{\upm{w,i}}{\downm{e,i}}+
    \sum_{\substack{ 1\leq i\leq z_w\\  s(e)=w
      }} \pi(s_e^*) \EE{\upm{e,i}}{\downm{w,i}}\\
    &=&\psi\left(\sum_{\substack{ 1\leq i\leq z_w\\  s(e)=w
        }} s_e \EE{\upm{w,i}}{\downm{e,i}}+
      \sum_{\substack{ 1\leq i\leq -z_w\\  s(e)=w
        }} s_e^* \EE{\upm{e,i}}{\downm{w,i}}\right)\\
    &=& \psi\left(\VV[\begin{lbmatrix}\zz\\
      0 \end{lbmatrix}]\right)
  \end{eqnarray*}
  since $s_e\in \III_{(H,S)}=\ker\pi$ when $s(e)$ (and thus $r(e)$) lies
  in $H$, and $z_w=x_w$ when $w\notin H$. 
  A similar computation for $\PP[\xx]$ shows
  that $\pi(\PP[\xx])=\psi\left(\PP[\begin{lbmatrix}\zz\\0 \end{lbmatrix}]\right)$. Thus 
  $\pi(\UU[\xx])=\psi\left(\UU[\begin{lbmatrix}\zz\\0 \end{lbmatrix}]\right)$ and
  \begin{equation*}\pi_*\circ\chi_1(\xx)=\bigl[\pi(\UU[\xx])\bigr]_1
  =\left[\psi(\UU[\begin{lbmatrix}\zz\\
    0 \end{lbmatrix}])\right]_1
  =\chi''\bigl(\begin{lbmatrix}\zz\\
    0 \end{lbmatrix}\bigr)
  =\chi_1''\circ I_1\circ P_2(\xx).\end{equation*}
\end{proof}

\begin{claim} \label{claim:inj}
  $\iota_*\circ\chi_1'=\chi_1\circ I_1$.
\end{claim}

\begin{proof}
  Fix $\xx\in\ker \begin{lbmatrix}A^t-I\\\alpha^t\\0\end{lbmatrix}$. 
  This follows like in Claim \ref{ciii} by choosing the bijections
  \begin{gather*}
    \upm{\cdot}:\upindex[\begin{lbmatrix}\xx\\0\end{lbmatrix}]\to
    \{1,\dots,\hhh\}\qquad
    \upm{\cdot}:\upindex\to
    \{1,\dots,\hhh\}
    \shortintertext{and}
    \downm{\cdot}:\downindex[\begin{lbmatrix}\xx\\0\end{lbmatrix}]\to \{1,\dots,\hhh\}\qquad
    \downm{\cdot}:\downindex\to \{1,\dots,\hhh\}
  \end{gather*}
  to be pairwise equal. 
\end{proof}

\begin{claim}
  $\partial_0=0$.
\end{claim}

\begin{proof}
  It follows from Claim \ref{claim:inj} that
  $\iota_*:K_1\bigl(\III_{(H,S)}\bigr)\to K_1\bigl(C^*(E)\bigr)$ is injective. Thus
  $\operatorname{im}(\partial_0)=0$ from which it follows that $\partial_0=0$.
\end{proof}

\begin{claim} 
  $\partial_1\circ\chi_1''=\chi_0'\circ \begin{lbmatrix}X^t&0\\\xi^t&0\\0&I\end{lbmatrix}$. 
\end{claim}

\begin{proof}
  Fix $\xx=\begin{lbmatrix}\yy\\\zz \end{lbmatrix}\in \ker\begin{lbmatrix}
    B^t-1&\Gamma^t\\\beta^t&\gamma^t\\\eta^t&Z^t-I\end{lbmatrix}$. 
  We lift $\psi(\VV[\xx])$ and $\psi(\PP[\xx])$ to 
  \begin{equation*}
    \VVt\ =\sum_{\substack{ 1\leq i\leq x_w\\
        \mathclap{s(e)=w, r(e)\notin H}}} s_e \,
    \EE{\upm{w,i}}{\downm{e,i}}\ +
    \sum_{\substack{ 1\leq i\leq -x_w\\  
        \mathclap{s(e) = w, r(e)\notin H}}} 
    s_e^* \, \EE{\upm{e,i}}{\downm{w,i}}
  \end{equation*} 
  and
  \begin{equation*}
    \PPt\ =\sum_{1\leq i\leq x_w} p_w\EE{\upm{w,i}}{\upm{w,i}}\ +
    \sum_{\substack{ 1\leq i\leq -x_w\\  
        \mathclap{s(e) = w, r(e) =v} \\  v\notin H}}
    p_v\EE{\upm{e,i}}{\upm{e,i}},
  \end{equation*}
  respectively, in $\matrM_{\hhh}\bigl(C^*(E)\bigr)$. 
  Since $\bigl\{s_e,p_v:e\in r^{-1}(E^0\setminus H),\ v\in E^0\setminus H\bigr\}$
  is a Toeplitz-Cuntz-Krieger $(E\setminus H)$-family, it follows from
  Lemma \ref{V-pi-lem} that $\VVt$ is a partial isometry and that
  $\PPt\VVt=\VVt\PPt=\VVt$. It follows that also $\UUt:=\VVt+(1-\PPt)$ is a
  partial isometry. Hence, to compute the value of the index map on
  $[\UU]_1$, we just need to compute the defect of $\UUt$ in
  $K_0(\III_{(H,S)})$, cf. \cite[Proposition 9.2.2]{mrflnl:ikc}. 
  We have, using Lemma \ref{foureqs}, that
  \begin{eqnarray*}
    1-\UUt\UUt^*&=&\PPt-\VVt\VVt^*\\
    &=&\sum_{1\leq i\leq x_w}\left(p_w-
      \sum_{\substack{\wev\\  v\not\in H}}
      s_es_e^*\right)\EE{\upm{w,i}}{\upm{w,i}}\\
    &=&\sum_{1\leq i\leq y_w}\left(\sum_{\wev} s_es_e^*-
      \sum_{\substack{\wev\\  v\not\in H}}
      s_es_e^*\right)\EE{\upm{w,i}}{\upm{w,i}}\\
    &&\quad + \sum_{1\leq i\leq z_{v_0}}\left(p_{v_0}-
      \sum_{  s(e)=v_0,r(e)\notin H} s_es_e^*\right)\EE{\upm{v_0,i}}{\upm{v_0,i}}\\
    &=&\sum_{1\leq i\leq y_w}
    \sum_{\substack{\wev\\  v\in H}}
    s_es_e^*\EE{\upm{w,i}}{\upm{w,i}}\
    + \sum_{1\leq i\leq z_{v_0}}p^H_{v_0} \EE{\upm{v_0,i}}{\upm{v_0,i}}.
\end{eqnarray*}
Passing to the $K_0$-group and using that $s_e\in C^*(E_H)$,  we get
\begin{eqnarray*}
  \bigl[1-\UUt\UUt^*\bigr]_0&=&\sum_{1\leq i\leq y_w}
  \sum_{\substack{\wev\\  v\in H}}
  [s_es_e^*]_0\
  +\sum_{1\leq i\leq z_{v_0}}\bigl[p^H_{v_0}\bigr]_0\\
  &=&\sum_{0<y_w} y_w
  \sum_{\substack{\wev\\  v\in H}}
  [s_e^*s_e]_0\
  +\sum_{0<z_{v_0}}z_{v_0}\bigl[p_{v_0}^H\bigr]_0
  \\
  &=&\sum_{0<y_w} y_w
  \sum_{\substack{\wev\\  v\in H}} [p_v]_0\
  +\sum_{0<z_{v_0}}z_{v_0}\bigl[p_{v_0}^H\bigr]_0.
\end{eqnarray*}
The computation for $1-\UUt^*\UUt$ is similar, and we get
\begin{equation*}
[1-\UUt\UUt^*]_0-
[1-\UUt^*\UUt]_0=
\sum_{0\not=y_w} y_w
\sum_{\substack{\wev\\  v\in H}} [p_v]_0\
+\sum_{z_{v_0}\ne 0}z_{v_0}\bigl[p_{v_0}^H\bigr]_0.
\end{equation*}
By comparison,
\begin{eqnarray*}
  \chi_0'\left(\begin{lbmatrix}X^t&0\\\xi^t&0\\0&I\end{lbmatrix}\xx\right)
  &=& \chi_0'\left(\begin{lbmatrix}X^t\yy\\\xi^t\yy\\\zz\end{lbmatrix}
    +\operatorname{im}\begin{lbmatrix}A^t-1\\\alpha^t\\0\end{lbmatrix}\right) \\
  &=&\sum_{y_w\not=0}y_w\left(\sum_{\substack{ v\in {E^0_\textnormal{reg} \cap H}
      }}X_{v,w}[p_v]_0 \ +\sum_{\substack{ v\in E^0_\textnormal{sing} \cap H
      }}\xi_{v,w}[p_v]_0\right)\\
  &&\quad+\sum_{z_{v_0}\ne 0}z_{v_0}\bigl[p_{v_0}^H\bigr]_0\\
  &=&\sum_{x_w\not=0}y_w\left(\sum_{\substack{\wev\\  v\in {E^0_\textnormal{reg} \cap H}
      }}[p_v]_0 \ +\sum_{\substack{\wev\\  v\in E^0_\textnormal{sing} \cap H
      }}[p_v]_0\right)\\
  &&\quad+\sum_{z_{v_0}\ne 0}z_{v_0}\bigl[p_{v_0}^H\bigr]_0\\
  &=&\sum_{x_w\not=0}y_w\sum_{\substack{\wev\\  v\in H
    }}[p_v]_0\ +\sum_{z_{v_0}\ne 0}z_{v_0}\bigl[p_{v_0}^H\bigr]_0,
\end{eqnarray*}
completing the proof.
\end{proof}


\begin{thebibliography}{10}

\bibitem{pamamep:nkga}
P.~Ara, M.~A. Moreno, and E.~Pardo.
\newblock Nonstable {$K$}-theory for graph algebras.
\newblock {\em Algebr. Represent. Theory}, 10(2):157--178, 2007.

\bibitem{tbdpirws:crg}
T.~Bates, D.~Pask, I.~Raeburn, and W.~Szymanski.
\newblock {$C^*$}-algebras of row-finite graphs.
\newblock {\em New York J. Math.}, 6:307--324, 2000.

\bibitem{lgbgkp:crrz}
L.G. Brown and G.K. Pedersen.
\newblock {$C^*$}-algebras of real rank zero.
\newblock {\em J.\ Funct.\ Anal.}, 99:131--149, 1991.

\bibitem{jc:hgsec}
J.~Cuntz.
\newblock On the homotopy groups of the space of endomorphisms of a {$C^{\ast}
  $}-algebra (with applications to topological {M}arkov chains).
\newblock In {\em Operator algebras and group representations, {V}ol. {I}
  ({N}eptun, 1980)}, volume~17 of {\em Monogr. Stud. Math.}, pages 124--137.
  Pitman, Boston, MA, 1984.

\bibitem{ddmt:ckegc}
D.~Drinen and M.~Tomforde.
\newblock Computing {$K$}-theory and {Ext} for graph {$C^*$}-algebras.
\newblock {\em Illinois J. Math.}, 46:81--91, 2002.

\bibitem{segrer:ccfis}
S.~Eilers, G.~Restorff, and E.~Ruiz.
\newblock Classifying {$C^*$}-algebras with both finite and infinite
  subquotients.
\newblock Preprint, arXiv:0801.0324v3, 2010.

\bibitem{semt:cnga}
S.~Eilers and M.~Tomforde.
\newblock On the classification of nonsimple graph algebras.
\newblock {\em Math. Ann.}, 346:393--418, 2010.

\bibitem{njfir:thb}
N.J. Fowler and I.~Raeburn.
\newblock The {T}oeplitz algebra of a {H}ilbert bimodule.
\newblock {\em Indiana Univ. Math. J.}, 48(1):155--181, 1999.

\bibitem{jaj:rrcag}
J.A Jeong.
\newblock Real rank of {$C^*$}-algebras associated with graphs.
\newblock {\em J. Aust. Math. Soc.}, 77(1):141--147, 2004.

\bibitem{tk:cac}
T.~Katsura.
\newblock On {$C^*$}-algebras associated with {$C^*$}-correspondences.
\newblock {\em J. Funct. Anal.}, 217(2):366--401, 2004.

\bibitem{psmmt:atc}
P.S. Muhly and M.~Tomforde.
\newblock Adding tails to {$C^*$}-correspondences.
\newblock {\em Doc. Math.}, 9:79--106, 2004.

\bibitem{wlp:kacgc}
W.~Paschke.
\newblock {$K$}-theory for actions of the circle group on {$C^*$}-algebras.
\newblock {\em J. Operator Theory}, 6(1):125--133, 1981.

\bibitem{mr:ccka}
M.~R{\o}rdam.
\newblock Classification of {C}untz-{K}rieger algebras.
\newblock {\em $K$-Theory}, 9(1):31--58, 1995.

\bibitem{mrflnl:ikc}
M.~R{\o}rdam, F.~Larsen, and N.~Laustsen.
\newblock {\em An introduction to {$K$}-theory for {$C^*$}-algebras}, volume~49
  of {\em London Mathematical Society Student Texts}.
\newblock Cambridge University Press, Cambridge, 2000.

\bibitem{mt:okgc}
M.~Tomforde.
\newblock The ordered {$K_0$}-group of a graph {$C^*$}-algebra.
\newblock {\em C. R. Math. Acad. Sci. Soc. R. Can.}, 25(1):19--25, 2003.

\end{thebibliography}
\end{document}